\documentclass[english]{amsart}
\usepackage{amsfonts, amsmath, amssymb, amscd, amsthm, array, graphicx}
\usepackage[latin9]{inputenc}
\usepackage{amsmath}
\usepackage{stmaryrd}
\usepackage{graphicx}
\usepackage{amssymb}
\usepackage{dsfont}
\usepackage{babel}
\usepackage{color}
\usepackage{cancel}
\usepackage[all]{xy}

\usepackage{mathtools}
\usepackage{hyperref,url}

\makeatletter
\def\blfootnote{\xdef\@thefnmark{}\@footnotetext}
\makeatother
\usepackage{geometry}
\usepackage{pdflscape}
\usepackage{graphicx}
\usepackage{float}

\usepackage{tikz}
\usetikzlibrary{babel}
\usepackage{tikz-cd}
\usetikzlibrary{calc}
\usetikzlibrary{cd}
\usetikzlibrary{automata,positioning,decorations.pathreplacing,decorations.markings,bending}
\usetikzlibrary{decorations.pathmorphing}
\usetikzlibrary{arrows}
\usetikzlibrary{shapes.geometric}
\usetikzlibrary{fit}

\tikzset{every state/.style={
    inner sep=1pt,
    minimum size=4pt,
    }}
\tikzset{every edge/.append style={
    line width=rule_thickness,          
    font=\scriptsize,    
    inner sep =.2em,
}}

\newlength\Textht
\newtheorem{thm}{Theorem}[section]
\newtheorem{cor}[thm]{Corollary}

\newtheorem{prop}[thm]{Proposition}
\newtheorem{que}[thm]{Question}

\theoremstyle{definition}
\newtheorem{dfn}[thm]{Definition}

\newtheorem{ex}[thm]{Example}

\newcommand{\N}{\mathbb{N}}
\newcommand{\Z}{\mathbb{Z}}
\newcommand{\F}{\mathbb{F}}

\newcommand{\schreier}[1]{\operatorname{Sch}(#1)}
\newcommand{\Schreier}[1]{\operatorname{Sch}\big( #1 \big)}
\newcommand{\cayley}[1]{\operatorname{Cay}(#1)}
\newcommand{\Cayley}[1]{\operatorname{Cay}\big( #1 \big)}

\newcommand{\GL}{\operatorname{GL}}
\newcommand{\Aut}{\operatorname{Aut}}

\newcommand{\End}{\operatorname{End}}

\newcommand{\im}{\operatorname{Im}}

\newcommand{\leqfg}{\leqslant_{\rm f.g.}}
\newcommand{\leqfi}{\leqslant_{\rm f.i.}}

\newcommand{\gen}[1]{\left\langle#1\right\rangle}

\newcommand{\ord}{\operatorname{ord}}

\begin{document}

\title[Computing $H$-equations with 2-by-2 integral matrices]{Computing {\boldmath $H$}-equations with 2-by-2 integral matrices}

\author{Gemma Bastardas}
\address{Departament de Matem\`atiques, Universitat Polit\`ecnica de Catalunya, CATALONIA.} \email{gemma.bastardas@upc.edu}

\author{Enric Ventura}
\address{Departament de Matem\`atiques and IMTech, Universitat Polit\`ecnica de Catalunya, CATALONIA.} \email{enric.ventura@upc.edu}



\begin{abstract}
We study the transference through finite index extensions of the notion of equational coherence, as well as its effective counterpart. We deduce an explicit algorithm for solving the following algorithmic problem about size two integral invertible matrices: ``given $h_1,\ldots ,h_r; g\in \operatorname{PSL}_2(\Z)$, decide whether $g$ is algebraic over the subgroup $H=\gen{h_1,\ldots ,h_r}\leqslant \operatorname{PSL}_2(\Z)$ (i.e., whether there exist a non-trivial $H$-equation $w(x)\in H*\gen{x}$ such that $w(g)=1$) and, in the affirmative case, compute finitely many such $H$-equations $w_1(x),\ldots ,w_s(x)\in H*\gen{x}$ further satisfying that any $w(x)\in H*\gen{x}$ with $w(g)=1$ is a product of conjugates of $w_1(x),\ldots ,w_s(x)$". The same problem for square matrices of size 4 and bigger is unsolvable.  
\end{abstract}

\maketitle

\section{Introduction}\label{intro}

In classical algebra one of the most basic topics of interest consists in studying \emph{polynomials} over a field $K$, namely $K[x]$, and their solutions in $K$ itself or in some field extension $K\leqslant L$. An element $\alpha \in L$ is said to be \emph{algebraic over $K$} if there exist a non-trivial polynomial with coefficients in $K$, say $0\neq p(x)\in K[x]$, such that $p(\alpha )=0$; otherwise, $\alpha$ is said to be \emph{transcendental} over $K$. The set of all polynomials over $K$ annihilating $\alpha$, namely $I_K(\alpha)=\{p(x)\in K[x] \mid p(\alpha)=0\}\leqslant K[x]$ form an ideal in $K[x]$ (which is non-trivial if and only if $\alpha$ is algebraic over $K$). Since the ring of polynomials $K[x]$ is a principal ideals domain, $I_K(\alpha)$ is generated by a single polynomial $I_K(\alpha)=m_{\alpha}(x)K[x]$; the unique monic such generator $m_{\alpha}(x)$ is called the \emph{minimal polynomial} of $\alpha$ over $K$. An interesting algorithmic problem is then to compute $m_{\alpha}(x)$ in terms of $\alpha$ (this includes deciding whether $\alpha\in L$ is algebraic or transcendental over $K$). 

Analogous questions can be asked and studied in the context of non-commutative Group Theory, where answers and algorithms turn out to be much more complicated. Here, we start with an \emph{extension of groups} $H\leqslant G$ and an element $g\in G$. The analog of the ``ring of polynomials over $K$" is the free product $H*\gen{x}\simeq H*\Z$, and the analog of the ``ideal of polynomials annihilating $\alpha$" is the normal subgroup $I_H(g) \unlhd H*\gen{x}$ of all ``polynomials" $w(x)\in H*\gen{x}$, such that $w(g)=1$. 

Let $G$ be a group and $H\leqslant G$. A univariate \emph{equation over $H$} (or an \emph{$H$-equation}, for short) is a ``polynomial equation" of the form $w(x)=1$, where $w(x)\in H*\gen{x}$, the free product of $H$ and the free (abelian) group of rank 1 generated by the \emph{variable} $x$. Written in its free product canonical form, $w(x)$ is an expression of the form
 \begin{equation}\label{eq:polynomial}
w(x)=h_0 x^{\varepsilon_1}h_1x^{\varepsilon_2}\cdots h_{d-1}x^{\varepsilon_d}h_d,
 \end{equation}
where $\varepsilon_i=\pm 1$, $h_i\in H$, and it is \emph{reduced}, namely  $\varepsilon_{i}=\varepsilon_{i+1}$ whenever $h_i=1$, for $i=1,\ldots ,d-1$. The \emph{degree} of such an $H$-equation is $d=\sum_{i=1}^{d}\vert \varepsilon_i \vert$ (when written in reduced form). It is \emph{balanced} if $\sum_{i=1}^{d} \varepsilon_i =0$. \emph{Constant} $H$-equations are those of degree 0, i.e., $h_0\in H$ and, among them, the \emph{trivial} one is $1\in H\leqslant H*\gen{x}$. It is standard to collect together into higher exponents the various possible consecutive occurrences of $x$ with trivial elements in between; for example, $h_1 x^2 h_2 x^{-2}$ stands for $h_1 x 1 x h_2 x^{-1} 1 x^{-1}$, a balanced equation of degree 4. 

An element $g\in G$ is a \emph{solution} to the $H$-equation $w(x)=1$ if $w(g)=_G1$, that is, if the element of $G$ resulting in substituting $x$ for $g$ in the expression~\eqref{eq:polynomial} is the trivial group element.

The area of Group Theory studying equations and their solutions is known as \emph{algebraic geometry over groups}; see~\cite{BMV99}. Following the language from this reference, a normal subgroup $I\unlhd H*\gen{x_1,\ldots,x_n}$ is called an \emph{ideal}, and the set of its common \emph{zeros} (i.e., the subset of $G^n$ of common solutions to the set of equations $\{ w(x)=1 \mid  w(x) \in I\}$) is an \emph{algebraic set}. There is a vast body of literature about the general problem of solving equations (univariate or multivariate) in a group $G$; see~\cite{BMV99, BDM09, Mak82, Raz87, Raz94} and the references therein. A typical question is to decide, given an equation (or a system of equations), whether it has a solution in $G$ or not and, in the positive case, to describe the set of all such solutions and its structure. Two of the most influential results in this direction are Makanin's and Razborov's theorems, which analyze the case of free groups: the first one (see Makanin~\cite{Mak82}) solves the decidability part of the problem, while the second one (see Razborov~\cite{Raz87, Raz94}) provides a kind of compact algorithmic description of \emph{all} solutions, in case they exist. These are two very deep results, with intricate proofs, and having numerous important applications.

The dual problem started being studied in groups much more recently. Given a group extension $H\leqslant G$ and an element $g\in G$, one says that $g$ is \emph{algebraic over $H$} (\emph{$H$-algebraic}, for short) if there exists a non-trivial $H$-equation $w(x)\in H*\gen{x}$ such that $w(g)=_G 1$; otherwise, $g$ will be called \emph{transcendental over $H$} (\emph{$H$-transcendental}, for short). In~\cite{Ros01} and~\cite{RV24} (see also the preliminary version~\cite{RV21}), the authors introduced these basic notions (under the name of \emph{$H$-dependence} instead of $H$-algebraicity), and studied ideals of $H$-equations in the framework of free groups $\F_n$. In this case, for every finitely generated subgroup $H\leqslant \F_n$ and any $g\in \F_n$, the ideal $I_H(g)$ is always finitely generated as normal subgroup of $H*\gen{x}$. Simultaneously, D. Ascari~\cite{AsTesi, As22} also considered similar questions and further studied the distribution of degrees of the equations in $I_H(g)$, again in the context of free groups. However, when considering the same notions in a general group $G$, the situation may be more complicated because the ideals $I_H(g)$ may very well be not finitely generated, even as normal subgroups of $H*\langle x\rangle$. 

In the classical situation with fields, $K\leqslant L$, the set of $K$-algebraic elements in $L$ form an intermediate field. In the context of groups this is far from true (see the following Example~\ref{ex:depH}(4), or the explicit examples given in~\cite{RV24} for the free group). However, it is straightforward to see that, in general, if $g\in G$ is algebraic over $H\leqslant G$ then so are all the elements in the double coset $HgH$; see~\cite[Obs~1.5]{RV24}. 

\begin{ex}\label{ex:depH}
\begin{enumerate}
\item If $g\in H\leqslant G$, then $g$ is $H$-algebraic, satisfying the obvious degree 1 $H$-equation $g^{-1}x=1$ (also $xg^{-1}=1$). 
\item If $H\leqslant G$ and $g\in G\setminus H$ satisfy $H^g\cap H\neq 1$, then $g$ is $H$-algebraic, satisfying the balanced degree 2 equation $x^{-1}hxh'^{-1}=1$, where $h,h'\in H$ are such that $g^{-1}hg=h'$. It follows that every $g\in G$ is algebraic over any nontrivial normal subgroup $\{1\}\neq H\unlhd G$. In particular, this is the case when $H$ is a non-trivial subgroup of the center $Z(G)$ of $G$.
\item If $H$ is a subgroup of finite index in $G$, then any $g\in G$ is $H$-algebraic: in fact, it satisfies the $H$-equation $h^{-1}x^k=1$, where $k\geqslant 1$ and $h\in H$ are such that $g^k=h$.
\item The set of elements $\{1\}$-algebraic are, precisely, the set of torsion elements of $G$: this follows from the fact that the only equations over the trivial subgroup are $\{1\}*\gen{x}=\{ x^n \mid n\in \Z \}$. \qed
\end{enumerate}
\end{ex}

Let us now look at $H$-equations from another point of view. When $H\leqslant G$ and $g\in G$, we can consider the \emph{evaluation} homomorphism
 \begin{eqnarray*}
\varphi_{H,g} \colon H*\gen{x} & \to & G, \nonumber \\ h & \mapsto & h, \; \forall h\in H, \\ x & \mapsto & g, \nonumber
 \end{eqnarray*}
which is well defined by the universal property of free products. Then, $(w(x))\varphi_{H,g} =w(g)$. So, in this setting, $g$ is a solution to $w(x)=1$ if and only if $w(x)\in \ker \varphi_{H,g}$. Thus, the set of $H$-equations satisfied by $g$ is, precisely, $I_H(g)=\ker \varphi_{H,g}$, of course, a normal subgroup of $H*\gen{x}$. For later use we introduce the following extra notation: for any other given subgroup $F\leqslant G$, the preimage of $F$ by $\varphi_{H,g}$ is denoted by $I_H(g;F)$. Clearly, 
 $$
I_H(g)=I_H(g;\{1\}) \leqslant I_H(g;F)=(F)\varphi_{H,g}^{-1}=\{w(x)\in H*\gen{x} \mid w(g)\in F\}\leqslant H*\gen{x}.
 $$

In the present paper we set up the situation for the general case of an arbitrary ambient group $G$ and see, as a first result, that imposing finite generation for all these ideals of equations (as normal subgroups) is equivalent to coherence of the ambient group $G$. Then, we introduce the algorithmic version of this notion (effective coherence), and revisit part of the existing literature on the topic; see Section~\ref{coherencia}. After that, in Section~\ref{finite-index}, we consider short exact sequences and see how these notions behave under taking finite index subgroups. Then, in Section~\ref{matrius} we apply our results to 2-by-2 matrices over $\Z$ and we see the unsolvability of the corresponding result for square matrices of size 4 and bigger. 

\subsection*{General notation and conventions}

For a group $G$, $\Aut(G)$ (resp., $\End(G)$) denotes the group (resp., the monoid) of automorphisms (resp., endomorphisms) of $G$. We write them all with the argument on the left, that is, we denote by $(x)\varphi$ (or simply $x\varphi$) the image of the element $x$ by the homomorphism $\varphi$; accordingly, we denote by  $\varphi \psi$ the composition $A \xrightarrow{\varphi} B \xrightarrow{\psi} C$. We will denote by $\GL_m(\Z)$ the group of size $m$ invertible matrices over the integers. 

Throughout the paper we write $H\leqfg G$ (resp., $H\leqfi G$) to express that $H$ is a finitely generated subgroup (resp., a finite index subgroup) of $G$, reserving the symbol $\leq$ for inequalities among real numbers.

\section{Equivalence with classical coherence}\label{coherencia}

We remind the classical notion of coherence and its algorithmic counterpart. 

\begin{dfn}
Let $G=\langle A \mid R\rangle$ be a finitely presented group. We say that $G$ is \emph{coherent} if every finitely generated subgroup $H\leqfg G$ is finitely presented. Moreover, we say that $G$ is \emph{effectively coherent} if $G$ is coherent and there is an algorithm which, on input $h_1,\ldots ,h_r\in G$ (given as words on $A$), computes a finite presentation for $H=\gen{h_1, \ldots ,h_r}\leqslant G$ on the given generators $h_1,\ldots ,h_r$.
\end{dfn}

A couple of comments are necessary here to make this definition more precise. A group $H$ being finitely presented means that it admits a presentation using finitely many generators and finitely many relations. However, the above definition requires to compute a finite presentation for $H$ (which is, certainly, finitely presented by the coherence hypothesis on the ambient $G$) \emph{on the given generators $h_1,\ldots ,h_r$}. This always makes sense by the following folklore result; see~\cite[Thm. 1.9]{Miller}.

\begin{prop}\label{miller}
Suppose that a finitely presented group $H$ is isomorphic to $H\simeq \langle A\mid R\rangle$, where $A$ is finite. Then, there is a finite subset $R_0\subseteq R$ such that, in the free group $\F(A)$, the normal closures of $R_0$ and $R$ do coincide, i.e., $\ll R_0 \gg_{\F(A)} =\ll R\gg_{\F(A)}$; in particular, $H\simeq \langle A\mid R\rangle =\langle A\mid R_0\rangle$. \qed
\end{prop}

The algorithmic requirement in the above definition is now clear: to compute finitely many relations among $h_1,\ldots ,h_r$, generating all the existing ones as normal subgroup. However, let us make a notational precision for later use. Consider the free group of rank $r$ as the free group on a set of $r$ formal variables, $\F_r=\F(\{x_1, \ldots ,x_r\})$. The group $H$ being (finitely) generated by $h_1,\ldots ,h_r$ means that the natural projection $\pi_H\colon \F_r \twoheadrightarrow H$, $x_1\mapsto h_1$, \ldots, $x_r\mapsto h_r$, is onto. And (after Proposition~\ref{miller}) the group $H$ being finitely presented means that, additionally, $\ker \pi_H$ is finitely generated as normal subgroup of $\F_r$; moreover, computing such a finite presentation means computing finitely many abstract words $w_1(\overrightarrow{x})=w_1(x_1,\ldots ,x_r),\ldots ,w_s(\overrightarrow{x})=w_s(x_1,\ldots, x_r)\in \F_r$ such that $H\simeq \langle x_1,\ldots ,x_r \mid w_1,\ldots ,w_s\rangle =\F_r/\ll w_1,\ldots ,w_s\gg$. Sometimes one abuses notation and just says that $H\simeq \langle h_1,\ldots ,h_r \mid R_1, \ldots ,R_s\rangle$ is a finite presentation for $H$, where $R_i=w_i(h_1,\ldots ,h_r)$, $i=1,\ldots ,s$. For the sake of clarity we prefer not to do so, reserving the notation $x_1,\ldots ,x_r$ and $w(\overrightarrow{x})$ when we refer to the elements up in the free group $\F_r$, and $h_1,\ldots ,h_r$ and $w(h_1,\ldots ,h_r)$ when we refer to their images by $\pi_H$ down in $H\leqslant G$. 

There are many examples of coherent groups in the literature: `small' groups like finite groups, abelian groups, or poly-cyclic groups; but also `big' groups like free groups, surface groups, fundamental groups of 3-manifolds (Scott~\cite{Sco73b}), or ascending HNN extensions of free groups (Feighn--Handel~\cite{FH99}), among others. It is also worth mentioning the characterization of coherence given by Droms~\cite{Dr87} among right-angled Artin groups: a RAAG is coherent if and only if its associated graph does not contain induced cycles of length greater than 3. Or the celebrated recent result by A. Jaikin-Zapirain and M. Linton in~\cite{JL24}, solving a famous conjecture by G. Baumslag: all 1-relator groups are coherent. Additionally, see the excellent recent survey on coherent groups by D. Wise~\cite{wise20}. 

The study of effective coherence started more recently. In~\cite[Prop. 6.1]{Gr99} Z. Grunschlag proved that torsion-free, hyperbolic, locally quasi-convex groups are effectively coherent, a result later extended by I. Bumagin and J. Macdonald in~\cite{BuMc16} to all groups discriminated by those groups. Later, P. Schupp showed that certain Coxeter groups are effectively coherent in~\cite{Sc03}, and I. Kapovich, R. Weidmann and A. Miasnikov  showed in~\cite{KWM05} that coherent RAAGs are, in fact, effectively coherent as well. Then, in~\cite{GrW09}, D. Groves and H. Wilton proved that groups with local retractions (i.e., where every finitely generated subgroup is a retract of a finite index subgroup) are effectively coherent; this includes free groups and limit groups (see also~\cite{BHMS}). Moreover, M. Bridson and H. Wilton considered in~\cite{BW11} the same algorithmic problem but in families of groups not being necessarily coherent, and requiring the tuple of words given as input \emph{to generate a finitely presented subgroup}; they provide interesting families of groups where this variation is unsolvable, even with the word problem being uniformly solvable. 

The first natural examples of effectively coherent groups (included in some of the above mentioned families) are free groups: by Nielsen-Schreier theorem, any (finitely generated) subgroup is again free and, in particular, finitely presented. And the first explicit examples of non-coherent groups are direct products of two or more free non-abelian groups. We remind elementary well-known proofs for these facts. 

\begin{ex}\label{ex-Nielsen}
Since subgroups of free groups are again free, free groups are clearly coherent. The algorithmic version of this fact, namely free groups being effectively coherent, is a folklore result: given $h_1,\ldots ,h_r\in \F_n$, one can apply a sequence of Nielsen transformations to the $r$-tuple of words $(h_1,\ldots, h_r)$ until getting $(1,\ldots ,1, h'_1,\ldots ,h'_s)$, with $0\leq s\leq r$, and $\{h'_1,\ldots ,h'_s\}$ being a free basis of $H=\gen{h_1,\ldots ,h_r}\leqslant \F_n$; in this situation, the sequence of 1's is giving us a complete set of $r-s$ relations among the original generators $h_1,\ldots ,h_r$. More precisely, consider the abstract free group $\F_r=\F(\{x_1, \ldots, x_r\})$ and the epimorphism $\pi\colon \F_r \twoheadrightarrow H\leqslant \F_n$, $x_1\mapsto h_1$, \ldots, $x_r\mapsto h_r$. Applying the same sequence of Nielsen transformations to the $r$-tuple $\overrightarrow{x}=(x_1,\ldots ,x_r)$, we obtain an alternative free basis $(w_1(\overrightarrow{x}),\ldots ,w_r(\overrightarrow{x}))$ for $\F_r$ mapping to $(1,\ldots ,1, h'_1,\ldots ,h'_s)$ by $\pi$. Since $(h'_1,\ldots ,h'_s)$ is a free basis for $H$, we have $\ker \pi =\gen{\gen{w_1(\overrightarrow{x}),\ldots, w_{r-s}(\overrightarrow{x})}}\lhd \F_r$, meaning that $\langle x_1,\ldots ,x_r \mid w_1(\overrightarrow{x}),\ldots ,w_{r-s}(\overrightarrow{x})\rangle$ is a presentation for $H$ on the given generators $h_1,\ldots ,h_r$. \qed  
\end{ex}

\begin{ex}
The typical example of non-coherent group is $\F_2\times \F_2$. This group is very important in algorithmic group theory because it was the first one known to have unsolvable membership problem. Given $n\geq 2$ and a finite presentation $H=\langle a_1,\ldots ,a_n \mid R_1,\ldots ,R_m \rangle$, K.A. Mihailova, in her influential paper~\cite{Mi58}, considered the subgroup of pairs of words in $\F_n$ determining the same element in $H$, 
 $$
M(H)=\{(w_1, w_2)\in \F_n \times \F_n \mid w_1 =_H w_2\}\leqslant \F_n \times \F_n.
 $$
She proved that its obvious elements $(a_1, a_1), \ldots ,(a_n,a_n), (1,R_1),\ldots ,(1,R_m)$ do generate $M(H)$; and, more interestingly, the membership problem for $M(H)$ within $\F_n \times \F_n$ is solvable (i.e., there exists an algorithm to decide whether a given $(w_1, w_2)\in \F_n \times \F_n$ belongs to $M(H)$) if and only if the word problem for $H$ is solvable.

In~\cite[Thm.~B]{Gr78}, F.J. Grunewald proved that if $H$ is infinite then $M(H)$ cannot be finitely presented, a result later generalized by other authors while describing the structure of finitely presented subgroups of $\F_n \times \F_n$. In particular, if $H$ is infinite, $M(H)$ is a finitely generated but non-finitely presented subgroup of $\F_n \times \F_n$; therefore, $\F_n \times \F_n$ is not coherent, for $n\geq 2$. \qed
\end{ex}

Another interesting example of a non-coherent group is Thomson's group $F$: it contains the lamplighter group $\Z \wr \Z$, which is finitely generated but not finitely presented; see~\cite[Cor. 20]{GS99}. 

\medskip 

As explained in the introduction, we are interested in algebraicity and transcendence of elements $g\in G$ with respect to finitely generated subgroups $H$ of a group $G$. We define the corresponding notions with the following suggestive names because they will happen to be equivalent with the corresponding notions about coherence (see Proposition~\ref{prop:equivalence} below). 


\begin{dfn}
Let $G=\langle A \mid R\rangle$ be a finitely presented group. We say that $G$ is \emph{equationally coherent} (\emph{eq-coherent}, for short) if, for every finitely generated subgroup $H\leqfg G$ and every element $g\in G$, the ideal $I_H(g)$ is finitely generated as normal subgroup of $H*\gen{x}$. Similarly, we say that $G$ is \emph{effectively eq-coherent} if $G$ is eq-coherent and there is an algorithm which, on input $h_1,\ldots ,h_r; g\in G$ (given as words on $A$), computes finitely many $H$-equations $w_1(x),\ldots ,w_s(x)\in H*\gen{x}$ such that $I_H(g)=\gen{\gen{w_1(x),\ldots ,w_s(x)}}\lhd H*\gen{x}$, where $H=\gen{h_1,\ldots ,h_r}\leqslant G$. 
\end{dfn}

\begin{ex}
Again, free groups are effectively eq-coherent: suppose we are given $h_1,\ldots ,h_r; g\in \F(A)$ (as words on $A$); after a standard calculation, we can assume that $\{h_1,\ldots ,h_r\}$ form a basis for $H=\gen{h_1,\ldots ,h_r}\leqslant \F(A)$. Now apply the algorithm provided in Example~\ref{ex-Nielsen} to $\{h_1,\ldots ,h_r, g\}$, and compute a presentation for $\gen{H, g}=\gen{h_1,\ldots ,h_r, g}\leqslant \F(A)$ on the generators $\{h_1,\ldots ,h_r, g\}$, say $\gen{H,g}\simeq \langle h_1,\ldots ,h_r, x \mid w_1(x),\ldots ,w_s(x)\rangle$. This means that $I_H(g)=\gen{\gen{w_1(x),\ldots ,w_s(x)}}\lhd H*\gen{x}=\F_{r+1}$. Observe, moreover, that each such non-trivial $w_i(x)$ is a genuine non-constant $H$-equation: each such relation mandatorily involves $g$ because $\{h_1,\ldots ,h_r\}$ were assumed to be freely independent. 

Moreover, an alternative algorithm was provided by A. Rosenmann and E. Ventura in~\cite{RV24}, with a much more interesting geometric flavor: construct the Stallings $A$-automaton for $H=\gen{h_1,\ldots ,h_r}\leqslant \F(A)$ and attach an extra petal to the basepoint spelling the word $g$; then keep applying elementary foldings until getting the Stallings $A$-automaton corresponding to the subgroup $\gen{H,g}\leqslant \F(A)$. There is a natural way to compute a non-constant $H$-equation $w(x)$ satisfying $w(g)=1$ from each \emph{closed} folding in this sequence (one where the two identified edges are parallel, i.e., they already have the same initial vertex and the same terminal vertex). Then, the authors show that the finitely many equations obtained in this way do generate $I_H(g)$ as normal subgroup of $H*\gen{x}$ (see~\cite{DV24} for a general survey on Stallings graph techniques for free groups). In fact, essentially the same algorithm shows that free groups are effectively coherent; see~\cite[Thm. 5.8]{DV24} for details.  \qed
\end{ex}

Note that effective eq-coherence implies solvability of the word problem \texttt{WP}$(G)$. In fact, a bit more: it implies solvability of the \emph{Order Problem} \texttt{OP}$(G)$: given $g\in G$ (as a word on the generators) compute its order, $\ord(g)=\min \{n\geq 1 \mid g^n=_G 1\}\in \N\cup \{\infty\}$.

\begin{prop}\label{OP}
Let $G=\langle A \mid R\rangle$ be a finite presentation. Then, 
 $$
G \text{ effectively eq-coherent } \quad \Rightarrow \quad \texttt{OP}(G) \text{ solvable } \quad \Rightarrow \quad \texttt{WP}(G) \text{ solvable}.  
 $$
Moreover, none of the converse implications hold. 
\end{prop}

\begin{proof}
Suppose $G$ is effectively eq-coherent. Given $g\in G$, we can apply the hypothesis to compute a finite presentation for the cyclic subgroup $H=\gen{g}\leqslant G$, say $H=\langle x \mid x^{n_1},\ldots ,x^{n_s}\rangle$. In this situation it is clear that $\ord(g)=\gcd(|n_1|,\ldots ,|n_s|)$, proving the first implication. The second is trivial ($g=_G 1$ if and only if $\ord(g)=1$).

The group $G=\F_2 \times \F_2$ serves as a counterexample for the first reciprocal implication: \texttt{OP}$(\F_2 \times \F_2)$ is clearly solvable (since $\F_2 \times \F_2$ is torsion-free), but $\F_2 \times \F_2$ is not coherent. A counterexample for the other implication is more subtle: such a group was first constructed by D. Collins in~\cite{Co73}. In this context, it is also worth mentioning the existence of non-coherent hyperbolic groups (so, with solvable word problem); see~\cite{Ri82}.
\end{proof}

\begin{cor}
The algorithm in the definition of effective eq-coherence further decides whether $g\in G$ is $H$-algebraic or $H$-transcendental.
\end{cor}

\begin{proof}
After computing generators for $I_H(g)=\gen{\gen{w_1(x),\ldots ,w_s(x)}}\lhd H*\gen{x}$, we can check whether each equation $w_i(x)$ is trivial or not, by using \texttt{WP}$(G)$ applied to its coefficients. Clearly, the element $g$ is $H$-transcendental if and only if $s=0$ or all of them are trivial as equations.
\end{proof}

We end this section by proving that these concepts (both, general and algorithmic versions) with and without equations are, in fact, equivalent. 

\begin{prop}\label{prop:equivalence}
Let $G=\langle A \mid R\rangle$ be a finite presentation. Then, 
\begin{itemize}
\item[(i)] $G$ is coherent if and only if $G$ is eq-coherent;
\item[(ii)] $G$ is effectively coherent if and only if $G$ is effectively eq-coherent.
\end{itemize} 
\end{prop}

\begin{proof}
Suppose $G$ is coherent (resp., effectively coherent). Given $h_1,\ldots ,h_r; g\in G$ (as words on $A$), consider (resp., compute) a finite presentation for $\gen{H,g}=\gen{h_1,\ldots ,h_r,g}\leqslant G$, say 
 $$
\gen{H,g}\simeq \Big\langle x_1,\ldots ,x_r, x_{r+1} \mid w_1(\overrightarrow{x}),\ldots ,w_s(\overrightarrow{x})\Big\rangle.
 $$
It is clear that $I_H(g)=\gen{\gen{w_1(h_1,\ldots ,h_r,x_{r+1}),\ldots ,w_s(h_1,\ldots ,h_r,x_{r+1})}}\lhd H*\gen{x_{r+1}}$. (Note that, by construction, $w_i(h_1,\ldots ,h_r, g)=_G 1$ and so, if $w_i(\overrightarrow{x})$ does not involve $x_{r+1}$, then $w_i(h_1,\ldots ,h_r)=_H 1$ and $w_i(h_1,\ldots ,h_r,x_{r+1})$ is the trivial equation from $H*\gen{x_{r+1}}$ and can be ignored.) This shows the two implications to the right.

Suppose now that $G$ is eq-coherent (resp., effectively eq-coherent). Given $h_1,\ldots ,h_r\in G$ (as words on $A$), we will show the existence (resp., will compute) a finite presentation for $H=\gen{h_1,\ldots ,h_r}\leqslant G$, by induction on $r$. If $r=0$ or $r=1$ this is obvious, since $H$ is trivial or cyclic in these cases (the computability of a specific presentation for the case $r=1$ follows from Proposition~\ref{OP}). Assume the result for $r\geq 1$ and consider the case $r+1$: $H=\gen{h_1, \ldots ,h_r, h_{r+1}}$. Write $H'=\gen{h_1, \ldots ,h_r}\leqslant H\leqslant G$, $\F_{r}=\F(\{x_1,\ldots ,x_r\})$, $\F_{r+1}=\F(\{x_1,\ldots ,x_r, x_{r+1}\})$, and consider the epimorphism $\pi\colon \F_{r+1}\twoheadrightarrow H$, $x_1\mapsto h_1$, \ldots , $x_r\mapsto h_r$, $x_{r+1}\mapsto h_{r+1}$, which factorizes in the obvious way as $\pi=(\alpha*Id)\beta$, 
 $$
\F_{r+1}=\F_r*\gen{x_{r+1}} \stackrel{\alpha*Id}{\,\,\,\twoheadrightarrow\,\,\,} H'*\langle x_{r+1}\rangle \stackrel{\beta}{\twoheadrightarrow} H.
 $$
Observe that, by induction hypothesis, $\ker\alpha$ is finitely generated (resp., and we can compute generators for it) as normal subgroup of $\F_{r}$; hence, the same is true for $\ker (\alpha*Id)=\gen{\gen{\ker\alpha}}\lhd \F_{r+1}$. Moreover, by hypothesis, $\ker \beta=I_{H'}(h_{r+1})$ is also finitely generated (resp., and we can compute generators for it) as normal subgroup of $H'*\gen{x_{r+1}}$. Therefore,  
 $$
\ker \pi=\gen{\gen{\ker (\alpha*Id),\, (\ker \beta)(\alpha*Id)^{-1}}}\lhd \F_{r+1}
 $$
is also finitely generated (resp., and we can compute generators for it) as normal subgroup of $\F_{r+1}$. This completes the proof.
\end{proof}

\section{Eq-coherence and finite index}\label{finite-index}

In this section we see that, as one may expect, coherence and effective coherence go up and down through finite index extensions. Even though being kind of folklore results, we haven't found explicit references in the literature. We give detailed proofs, with special emphasis on the algorithmic version, which will be applied later to the case of integral invertible two-by-two matrices along the next section. Let us begin by reminding some well-known concepts and facts. 

\begin{dfn}
Let $G=\gen{A}$ be a group, generated by the finite set of elements $A$, and let $F\leqslant G$ be a subgroup. The \emph{Schreier graph}, denoted $\schreier{F,G,A}$, is the oriented graph with vertex set $V=F\backslash G$ (the set of left cosets, $Fg$, of $G$ modulo $F$), edge set $E=V\times A$, and incidence functions $\iota, \tau\colon E\to V$ given by $\iota(Fg,a)=Fg$ and $\tau(Fg,a)=Fga$ (we allow parallel edges and loops). The second coordinate of an edge is called its \emph{label}, $\ell(Fg, a)=a\in A$, and $(Fg,a)$ is said to be an \emph{$a$-edge}. By iteratively reading $a\in A$ whenever crossing an $a$-edge ahead, and reading $a^{-1}$ whenever crossing it backwards, any path $p$ from the underlying undirected graph of $\schreier{F,G,A}$ determines a word on $A^{\pm}$ and so, an element $\ell(p)\in G$, called its \emph{label}. 
\end{dfn}

As a direct consequence from $A$ generating $G$, it is well known that (the underlying undirected graph of) $\schreier{F,G,A}$ is always connected, and $2|A|$-regular with exactly one $a$-edge going in, and exactly one $a$-edge going out, from every vertex, and every $a\in A$. It is also clear from the definition that $|V|=[G : F]$ so, $\schreier{F,G,A}$ is finite if and only if $F\leqfi G$. Moreover, any word on $A^{\pm 1}$, say $w=a_{i_1}^{\epsilon_1}\cdots a_{i_n}^{\epsilon_n}$, and any vertex $Fg\in V$, determine a path $p_{_{Fg}, w}$ on (the underlying undirected graph of) $\schreier{F,G,A}$ given by  
 $$
p_{_{Fg},w}\colon Fg \stackrel{a_{i_1}^{\epsilon_1}}{\longrightarrow} F(ga_{i_1}^{\epsilon_1}) \stackrel{a_{i_2}^{\epsilon_2}}{\longrightarrow} \cdots \stackrel{a_{i_n}^{\epsilon_n}}{\longrightarrow} F(ga_{i_1}^{\epsilon_1}\cdots a_{i_n}^{\epsilon_n})
 $$
(understanding $p\stackrel{a^{-1}}{\longrightarrow}q$ as $p\stackrel{a}{\longleftarrow} q$ crossed backwards, $a\in A$). Denote the terminal vertex of such path as $(Fg)\cdot w=F(gw)$, and note that $w\in F$ if and only if $F\cdot w=F$, i.e., if and only if $p_{_{F},w}$ is a closed path. 

Let $T\subseteq \schreier{F,G,A}$ be a maximal subtree. For every edge $e\in E\setminus ET$, we can define the element $w_e=\ell(T[F, \iota e]\, e\, T[\tau e, F])\in G$, where $T[p,q]$ denotes the $T$-geodesic from vertex $p$ to vertex $q$. It is straightforward to see that the set $\{w_e \mid e\in E\setminus ET\}\subseteq F$ generates $F$; see~\cite{DV24}. In particular, finite index subgroups of finitely generated groups are, again, finitely generated. 

We remind three classical results, including a constructive proof for Proposition~\ref{prop: contr Sch}, for later use (see an explicit example in Section~\ref{matrius}).

\begin{prop}[{Todd--Coxeter~\cite[Ch.~8]{J90}}]\label{prop: TC}
Let $G=\langle A \mid R\rangle$ be a finite presentation and let $F\leqfi G$ be a finite index subgroup given by a finite set of generators (as words on $A^{\pm}$). Then we can algorithmically compute the index $[G : F]$, a set of left coset representatives of $G$ modulo $F$, 
and draw the Schreier graph $\schreier{F,G,A}$. In particular, we can also solve the membership problem of $F\leqslant G$. \qed
\end{prop}

\begin{prop}\label{prop: contr Sch}
Let $G=\langle A \mid R\rangle$ be a finite presentation, $A=\{a_1,\ldots ,a_n\}$, and let $F\leqfi G$ be a finite index subgroup, from which we only assume to know an algorithm solving the membership problem $F\leqslant G$. Then, we can algorithmically compute the index $[G : F]$, a set of left coset representatives of $G$ modulo $F$, draw the Schreier graph $\schreier{F,G,A}$, and compute a finite set of generators for $F$ (even a finite presentation for $F$ as well).
\end{prop}

\begin{proof}
Note that we are only given a finite presentation for $G$, and an algorithm which, on input a word $g\in (A^{\pm})^*$, just answers \texttt{Yes}/\texttt{No} according to whether the corresponding element $g\in G$ belongs to the finite index subgroup $F\leqfi G$ or not; nothing else from $F$ is assumed to be known (even the actual index, which must be computed). The strategy is to keep drawing the successively growing balls $B_0, B_1, B_2,\ldots$ in the Schreier graph $\schreier{F,G,A}$ until having it completely depicted. (Define $B_r$ as the full subgraph of $\schreier{F,G,A}$ spanned by those vertices at distance up to $r$ from the vertex $F$.) 

Start with a single vertex $V=\{F\}$ and no edge $E=\emptyset$; to complete the drawing of $B_0$ we need to know whether some edge (or edges) are loops at $F$: the edge $(F, a_j^{\epsilon_j})$ terminates at $F$ if and only if $a_j^{\epsilon_j}\in F$, which can be decided using the membership problem assumption on $F\leqslant G$. To draw $B_1$, look at the (up to) $2n$ neighbors of $F$, namely $Fa_1^{\epsilon_1},\ldots ,Fa_n^{\epsilon_n}$, for $\epsilon_1, \ldots ,\epsilon_n=\pm 1$, drawn with the corresponding edges connecting them to $F$. In order to do this, we need the information on which coset $Fa_i^{\epsilon_i}$ is equal to which $Fa_j^{\epsilon_j}$ (and maybe to $F$); but this can be easily deduced from several applications of the algorithm in the hypothesis: $Fa_i^{\epsilon_i}=Fa_j^{\epsilon_j}$ if and only if $a_i^{\epsilon_i}a_j^{-\epsilon_j}\in F$. To complete the drawing of $B_1$ we need to know whether some extra edges connect two of the existing vertices. This is not going to happen involving the vertex $F$ from $B_0$ because we already have the $2n$ edges incident to it. And for each one of the new vertices, say $Fa_i^{\epsilon_i}$, and each $a_j^{\epsilon_j}\in A^{\pm}$, the edge $(Fa_i^{\epsilon_i}, a_j^{\epsilon_j})$ terminates at $Fa_{k}^{\epsilon_k}$ if and only if $a_i^{\epsilon_i} a_j^{\epsilon_j} a_k^{-\epsilon_k}\in F$. Iteratively repeating this process (which needs an exponentially growing number of calls to the algorithm for membership from the hypothesis), we can draw the bigger and bigger balls $B_2, B_3,\ldots$ of $\schreier{F,G,A}$ until arriving at a big enough $M>0$ such that all the existing vertices in $B_M$ have degree $2n$; at this point, $B_M =B_{M+1}=\schreier{F,G,A}$ and we are done. This is guaranteed to happen after finite time because, by hypothesis, $F$ has finite index in $G$ ---even without knowing this index in advance--- and so, $\schreier{F,G,A}$ is a finite graph.

Once we have computed the full Schreier graph $\schreier{F,G,A}$, the rest is easy: the index $[G : F]$ is its total number of vertices, and a set of left coset representatives of $G$ modulo $F$, is easily computable by just reading the label of an arbitrary path from $F$ to each vertex. Also, a finite set of generators for $H$ easily follows from selecting a maximal tree $T$ and computing the elements $w_e\in F\leqslant G$, for each $e\in E\setminus ET$. Moreover, a presentation for $F$ with respect to the obtained generators can be also computed using the classical Reidemeister--Schreier process; see~\cite{Bo08}.
\end{proof}



\begin{prop}\label{prop: finite index}
Let $G$ be a  group and $H\leqfi G$ a finite index subgroup. Then $G$ is finitely presented if and only if $H$ is finitely presented. 
\end{prop}

\begin{proof}
See~\cite[Ch.~9, Cor.~1]{J90} and~\cite[Ch.~10, Cor.~3]{J90}. 
\end{proof}

We conclude this section by seeing that coherence, and effective coherence, go up and down finite index steps. 

\begin{thm}\label{thm: main}
Let $G$ be a group and $F\leqfi G$ a finite index subgroup. Then, 
\begin{itemize}
\item[(i)] $F$ is coherent if and only if $G$ is coherent.
\end{itemize}
Moreover, if $G=\langle A \mid R\rangle$ is given by a finite presentation and $F\leqfi G$ by a finite set of generators as words on $A^{\pm}$ (or, equivalently, by an algorithm solving the membership problem $F\leqslant G$) then, 
\begin{itemize}
\item[(ii)] $F$ is effectively coherent if and only if $G$ is effectively coherent.
\end{itemize}
\end{thm}

\begin{proof}
The two implications to the left are clear from the definitions.

Assume now that $F$ is coherent, i.e., any finitely generated subgroup of $F$ is finitely presented. If $H\leqfg G$ then $H\cap F\leqfi H\cap G=H\leqfg G$; hence, $H\cap F\leqfg F$ and, by hypothesis, $H\cap F$ is finitely presented. By Proposition~\ref{prop: finite index}, since $H\cap F\leqfi H$, $H$ is finitely presented as well. This completes the proof of~(i).

To show (ii), we have to be more accurate with the algorithmic treatment. In the following lines, we present an argument avoiding the computation of the intersection $H\cap F$ and working, instead, with equations. This paves the land for the application given in the next section. By hypothesis, we are given a finite presentation for $G$, namely $G=\langle A \mid R\rangle=\langle a_1,\ldots ,a_n \mid r_1,\ldots ,r_m\rangle$, and finitely many words on $A^{\pm 1}$, say $f_1,\ldots ,f_r$, generating an effectively coherent finite index subgroup $F=\gen{f_1,\ldots ,f_r}\leqfi G$; in particular, we have an algorithm such that, given finitely many words on $f_1,\ldots ,f_r$, it computes a presentation for the subgroup of $F$ they generate (and with respect to the given generators). We will prove that $G$ is effectively coherent under the equivalent form of effective eq-coherence; see Proposition~\ref{prop:equivalence}(ii). 

As a preparation, let us see that, without loss of generality, we can additionally assume that $F\lhd G$ with $G/F$ being a known finite group; in other words, we can assume that, together with $F$, we are also given a short exact sequence $1\to F\to G\to L\to 1$, where $L$ is a known finite group. In fact, compute the Schreier graph $\schreier{F,G,A}$, using Todd--Coxeter Proposition~\ref{prop: TC} (in particular, we can solve membership of $F\leqslant G$). Then, look at the natural action of $G$ on its (finite!) set of vertices $V=F\backslash G$, 
 $$
\begin{array}{rcl} \cdot\,\, \colon G & \to & \operatorname{Sym}(V) \\ g & \mapsto & \!\!\! \begin{array}[t]{rcl} \cdot g \colon V & \to & V \\ Fx & \mapsto & Fxg. \end{array} \end{array}
 $$
By construction, this is nothing else but a group homomorphism $\cdot$ from $G$ to the finite symmetric group $\operatorname{Sym}(V)=S_N$, where $N=|V|$. Note that $g\in F \Leftrightarrow Fg=F \Leftrightarrow \cdot g$ fixes the vertex $F$. It is straightforward to see that $\ker(\cdot)=\{g\in G \mid \cdot g=Id\}$ is the so-called \emph{core of $F$}, $\ker(\cdot)=\operatorname{core}(F)=\cap_{g\in G} F^g$, namely the intersection of all conjugates of $F$ by elements from $G$ (note that there are finitely many of them because $F\leqfi G$ and so, $\operatorname{core}(F)\leqfi G$ again). Since membership of $\operatorname{core}(F)\leqslant G$ is clearly solvable (by checking whether $\cdot g$ moves some vertex or fixes them all), we can compute a finite set of generators for $\operatorname{core}(F)$, as words on $A^{\pm}$ (or, equivalently, on $\{ f_1,\ldots ,f_r\}$); see Proposition~\ref{prop: contr Sch}. Since effective coherence of $F$ implies effective coherence of $\operatorname{core}(F)$, we have completed the proof of the announced reduction (with $L$ being the finite group of permutations $L=\im(\cdot)\leqslant \operatorname{Sym}(V)$).  

Let us reset notation and assume we are given the finite presentation $\langle A \mid R\rangle$ for $G$, generators $f_1,\ldots ,f_r$ (as words on $A^{\pm}$) for a normal finite index subgroup $F\lhd_{f.i.} G$, and a group homomorphism $\pi\colon G\twoheadrightarrow L$ onto a finite group $L$ (where we can compute), satisfying $\ker (\pi)=F$; summarizing, we are given the short exact sequence
 $$
1\to F\to G=\langle A \mid R\rangle \stackrel{\pi}{\to} L\to 1.
 $$
For every $g\in G$, denote $\overline{g}=g\pi\in L$. 

Suppose we are also given several words on $A$, say $h_1,\ldots ,h_s;g$, and our pending goal is to compute a finite set of generators for $I_H(g)$ as normal subgroup of $H*\gen{x}$, where $H=\gen{h_1,\ldots ,h_s}\leqfg G$ (this will show that $G$ is effectively eq-coherent and so, effectively coherent). Since $I_{H\pi}(\overline{g})=\{w(x)\in H\pi*\gen{x} \mid w(\overline{g})=\overline{1}\}$ is the kernel of the homomorphism 
 $$
\begin{array}{rcl} H\pi*\gen{x} & \twoheadrightarrow & \gen{H\pi, \overline{g}}, \\ w(x) & \mapsto & w(\overline{g}), \end{array}
 $$
and since $\gen{H\pi, \overline{g}}\leqslant L$ is finite, we deduce that $I_{H\pi}(\overline{g})\leqfi H\pi*\gen{x}$. Moreover, we can effectively solve the membership problem in $I_{H\pi}(\overline{g})\leqslant H\pi*\gen{x}$ by just plugging $\overline{g}$ into $x$ and computing in $L$. Consider now the epimorphism $\varphi=\pi|_H*id \colon H*\gen{x}\twoheadrightarrow H\pi *\gen{x}$, $w(x)\mapsto \overline{w}(x)$, mapping any $H$-equation $w(x)$ to the corresponding $H\pi$-equation obtained by replacing each coefficient by its image under $\pi$, denoted $\overline{w}(x)$. And consider the full $\varphi$-preimage of $I_{H\pi}(\overline{g})$, namely, 
 $$
\begin{aligned}
\big( I_{H\pi}(\overline{g}) \big)\varphi^{-1} & = \big\{ w(x)\in H*\gen{x} \mid \overline{w}(x)\in I_{H\pi}(\overline{g}) \big\} \\ & = \big\{ w(x)\in H*\gen{x} \mid \overline{w}(\overline{g})=\overline{1} \big\} \\ & =\big\{ w(x)\in H*\gen{x} \mid \overline{w(g)}=\overline{1} \big\} \\ & = \big\{ w(x)\in H*\gen{x} \mid w(g)\in F \big\} \\ & =I_{H}(g;F).
 \end{aligned}
 $$
Since $\varphi$ is onto and $I_{H\pi}(\overline{g})\leqfi H\pi*\gen{x}$, we deduce that $I_{H}(g;F)\leqfi H*\gen{x}$ (in fact, with the same index). Moreover, we can also effectively solve the membership problem in $I_{H}(g;F)\leqslant H*\gen{x}$: just map the given equation down to $H\pi*\gen{x}$ and check whether it belongs to $I_{H\pi}(\overline{g})$. Therefore, by Proposition~\ref{prop: contr Sch}, we can compute a finite set of generators for $I_{H}(g;F)$, say $I_{H}(g;F)=\gen{w_1(x),\ldots ,w_p(x)}$. These are $H$-equations $w_i(x)$ such that $w_i(g)\in F$ (note that $H\cap F\leqslant I_H(g)$, viewed as constant equations). 

Our target is inside this subgroup, $I_H(g)\leqslant I_H(g;F)\leqfi H*\gen{x}$. Compute the elements $v_i=w_i(g)\in F$, for $i=1,\ldots ,p$. More precisely: during the last application of Proposition~\ref{prop: contr Sch}, we have computed $\schreier{I_H(g;F),\, H*\gen{x},\, \{h_1,\ldots ,h_s,x\}}$ and have obtained $w_1(x),\ldots ,w_p(x)\in H*\gen{x}$ as words on $(\{h_1,\ldots ,h_s,x\})^{\pm}$, which then can be recomputed as words on $(A\cup\{x\})^{\pm}$; plug $g$ (a word on $A^{\pm}$) into $x$ and rewrite the results $v_1=w_1(g),\ldots ,v_p=w_p(g)$ (which are words on $A^{\pm}$) as words on $\{f_1,\ldots ,f_r\}$ (this can be done by a standard brute force algorithm, knowing in advance that $v_1,\ldots ,v_p\in F$). Now it is time to apply the hypothesis on $F$: compute a presentation for $V=\gen{v_1,\ldots ,v_p}\leqslant F$ with respect to the given generators, say $V\simeq \langle x_1,\ldots ,x_p \mid s_1,\ldots ,s_q \rangle$. This means that $s_i=s_i(x_1,\ldots ,x_p)$ are abstract words such that $s_i(v_1,\ldots ,v_p)=1$, and that any word $s(x_1,\ldots ,x_p)$ satisfying $s(v_1,\ldots ,v_p)=1$ equals a product of conjugates of them, 
 $$
s(x_1,\ldots ,x_p)=\prod_{j=1}^n z_j (x_1,\ldots ,x_p)^{-1} \big( s_{i_j}(x_1,\ldots ,x_p)\big)^{\epsilon_{j}}z_j (x_1,\ldots ,x_p).
 $$
Then, we have that 
 $$
\begin{aligned}
I_H(g) & =\{w(x)\in H*\gen{x} \mid w(g)=1 \} \\ & =\{w(x)\in I_H(g;F) \mid w(g)=1\} \\ & =\Big\{ \prod_{j=1}^n w_{i_j}(x)^{\epsilon_j} \mid \prod_{j=1}^n w_{i_j}(g)^{\epsilon_j}=1 \Big\} \\ & =\Big\{ \prod_{j=1}^n w_{i_j}(x)^{\epsilon_j} \mid \prod_{j=1}^n v_{i_j}^{\epsilon_j}=1 \Big\} \\ & =\gen{\gen{s_1 \big( w_1(x), \ldots ,w_p(x)\big),\ldots ,s_q \big( w_1(x), \ldots ,w_p(x)\big) }} \lhd H*\gen{x},
\end{aligned}
 $$
where the justification for the last equality is as follows: the inclusion to the left is clear because $s_i \big( w_1(g), \ldots ,w_p(g) \big) =s_i(v_1, \ldots ,v_p)=1$ and so, $s_i \big( w_1(x), \ldots ,w_p(x) \big) \in I_H(g)$, for $i=1,\ldots ,q$. For the other inclusion, consider a product of the form $w(x)=\prod_{j=1}^n w_{i_j}(x)^{\epsilon_j}$ such that $\prod_{j=1}^n v_{i_j}^{\epsilon_j}=1$. Since the formal product $s(x_1,\ldots ,x_p)=\prod_{j=1}^n x_{i_j}^{\epsilon_j}$ satisfies $s(v_1,\ldots ,v_p)=\prod_{j=1}^n v_{i_j}^{\epsilon_j}=1$, the previous paragraph tells us that it must be of the form
 $$
s(x_1,\ldots ,x_p)=\prod_{j=1}^n z_j (x_1,\ldots ,x_p)^{-1} \big( s_{i_j}(x_1,\ldots ,x_p)\big)^{\epsilon_{j}}z_j (x_1,\ldots ,x_p).
 $$
Replacing $x_i$ by $w_i(x)$, $i=1,\ldots ,p$, we obtain
 \begin{align*}
w(x) & =s\big( w_1(x),\ldots ,w_p(x) \big) = \\ & =\prod_{j=1}^n z_j \big( w_1(x),\ldots ,w_p(x) \big)^{-1} \big( s_{i_j}\big( w_1(x),\ldots ,w_p(x)\big) \big)^{\epsilon_{j}}z_j \big( w_1(x),\ldots ,w_p(x)\big),
 \end{align*}
which means that 
 $$
w(x)\in \gen{\gen{s_1 \big( w_1(x), \ldots ,w_p(x)\big),\ldots ,s_q \big( w_1(x), \ldots ,w_p(x)\big)}}\lhd H*\gen{x},
 $$
as we wanted to see. This completes the proof. 
\end{proof}

\section{The case of integral matrices}\label{matrius}

Let us apply the previous results to the usual groups of integral 2-by-2 matrices. A lot is known about the structure of groups like $\operatorname{GL}_2(\Z)$, $\operatorname{PGL}_2(\Z)$, $\operatorname{SL}_2(\Z)$, or $\operatorname{PSL}_2(\Z)$. They are all virtually-free with well known normal free subgroups whose corresponding quotients are also well known finite groups; see, for example, \cite[Ch.~1, Exs.~5.2]{DD89} for the structure and specific sets of generators. We note that virtual freeness is a very specific property for matrices of size 2: each of these integral matrix groups in size 3 and bigger are very far from being virtually free. 

For virtually free groups, many algorithmic properties are known to be solvable: the word, conjugacy and isomorphism problems, the double-subgroup conjugacy problem, the order and relative order problems, the membership problem, the intersection problem, etc (all of them in their standard forms, where the group is given by generators and relations). Of course, all these solutions apply also to the above mentioned 2-by-2 integral matrix groups. If one makes the appropriate translation from the formalism of presentations to the plain language of integral matrices (namely, two-by-two arrays of integral numbers), one obtains algorithms working with real matrices. Most of these algorithms perform tasks which are not easy looking from the plain matrix point of view. As a very illustrative example, we can mention the well-known membership problem:

\begin{thm}
Let $G$ be one of the groups $\operatorname{GL}_2(\Z)$, $\operatorname{PGL}_2(\Z)$, $\operatorname{SL}_2(\Z)$, or $\operatorname{PSL}_2(\Z)$. There exists an algorithm which, given matrices $h_1,\ldots ,h_s;g$ from $G$, decides whether $g$ can be expressed as a product of the form $g=h_{i_1}^{\epsilon_1}\cdots h_{i_n}^{\epsilon_n}$, for some $n\geq 0$, $i_1,\ldots ,i_n\in \{1,\ldots ,s\}$, and $\epsilon_1,\ldots ,\epsilon_n=\pm 1$; furthermore, in the affirmative case, the algorithm outputs such an expression. \qed
\end{thm}

This result is very well known for free and virtually-free groups. Modern adaptations to the matrix context can be found, for example, in~\cite{L21} and~\cite{CT24}, including detailed analysis of the complexity of the provided algorithms. The goal of this section is to present another of such applications, illustrating the effective eq-coherence of these matrix groups, a direct corollary of our Theorem~\ref{thm: main}(ii) above, fully translated in terms of plain integral matrices. 

\begin{cor}\label{cor: main matrius}
Let $G$ be one of $\operatorname{GL}_2(\Z)$, $\operatorname{PGL}_2(\Z)$, $\operatorname{SL}_2(\Z)$, or $\operatorname{PSL}_2(\Z)$. There is an algorithm which, given matrices $h_1,\ldots ,h_s;g\in G$, decides whether $g$ is algebraic over the subgroup $H=\gen{h_1,\ldots ,h_s}\leqslant G$; moreover, in the affirmative case, the algorithm computes finitely many $H$-equations $w_1(x),\ldots ,w_p(x)\in I_H(g)\lhd H*\gen{x}$ further satisfying that any $w(x)\in H*\gen{x}$ with $w(g)=1$ is a product of conjugates of $w_1(x),\ldots ,w_p(x)$. \qed
\end{cor}

Let us concentrate on the group $G=\operatorname{PSL}_2(\Z)=\operatorname{SL}_2(\Z)/\{\pm I\}$, namely 2-by-2 integral matrices with determinant 1, modulo its center; here, $I=\left( \begin{smallmatrix} 1 & 0 \\ 0 & 1 \end{smallmatrix} \right)$ stands for the 2-by-2 identity matrix. We make the usual abuse of language of working with real matrices to represent elements of $\operatorname{PSL}_2(\Z)$ (which really are equivalence classes of matrices modulo $\pm I$); hence, $-I=I$ in $\operatorname{PSL}_2(\Z)$. This group is isomorphic to the free product of the cyclic groups of orders 2 and 3, 
 $$
\operatorname{PSL}_2(\Z)\simeq C_2*C_3=\langle a,b \mid a^2, b^3\rangle,
 $$
and we can take as generators the matrices $a=\left( \begin{smallmatrix} 0 & -1 \\ 1 & 0 \end{smallmatrix} \right)$ and $b=\left( \begin{smallmatrix} 1 & -1 \\ 1 & 0 \end{smallmatrix} \right)$; see, for example, \cite[Ch.~1, Exs.~5.2]{DD89}. Clearly, $a^2=b^3=-I=I$. 

Abelianizing, we obtain $(\operatorname{PSL}_2(\Z))^{\rm ab}\simeq C_2\times C_3\simeq C_6$ and the abelianization map sends $a$ to $(1,0)$ and $b$ to $(0,1)$. More precisely, we have the short exact sequence
 \begin{equation}\label{eq: ses}
\begin{array}{ccccccccc}
1 & \longrightarrow & F & \longrightarrow & \operatorname{PSL}_2(\Z) & \stackrel{\pi}{\longrightarrow} & C_2\times C_3 & \longrightarrow & 1. \\ & & & & a & \mapsto & (1,0) & & \\ & & & & b & \mapsto & (0,1) & & 
\end{array}
 \end{equation}
By normality of $F=\ker \pi=[\operatorname{PSL}_2(\Z), \operatorname{PSL}_2(\Z)]$ in $\operatorname{PSL}_2(\Z)$, the Schreier graph of $F$ coincides with the Cayley graph of the quotient, 
 $$
\Schreier{F,\operatorname{PSL}_2(\Z), \{a,b\}} =\Cayley{\operatorname{PSL}_2(\Z)/F, \{a\pi, b\pi\}}= \Cayley{C_2\times C_3, \{(1,0), (0,1)\}}.
 $$
With elementary calculations in $C_2\times C_3$, we can depict $\cayley{C_2\times C_3, \{(1,0), (0,1)\}}$ and, using the set of coset representatives $\{I, b, b^2, a, ab, ab^2\}$, we get the (isomorphic) $\Schreier{F,\operatorname{PSL}_2(\Z), \{a,b\}}$ depicted in Figure~\ref{fig: sch F}.

\begin{figure}[h] 
\centering
  \begin{tikzpicture}[>=stealth, auto, node distance=1cm]
  
   \node[] (Fa) {$F a$};
   \node[] (Fab) [above left = 0.5 and 1 of Fa] {$Fab$};
   \node[] (Fabb) [above right = 0.5 and 1 of Fa] {$Fab^2$};
   \node[] (F) [below = of Fa] {$F$};
   \node[] (Fb) [below left = 0.5 and 1.175 of F] {$Fb$};
   \node[] (Fbb) [below right = 0.5 and 1.175 of F]{$Fb^2$};

   \path[<->] (Fa) edge[blue]
             node[right] {\scriptsize{$a$}}
             (F);

    \path[->] (Fa) edge[red]
             node[above right] {\scriptsize{$b$}}
             (Fab);

    \path[->] (Fabb) edge[red]
             node[above left] {\scriptsize{$b$}}
             (Fa);

    \path[->] (F) edge[red]
             node[above left] {\scriptsize{$b$}}
             (Fb);

    \path[->] (Fbb) edge[red]
             node[above right] {\scriptsize{$b$}}
             (F);

    \path[->] (Fab) edge[red]
             node[above] {\scriptsize{$b$}}
             (Fabb);

    \path[->] (Fb) edge[red]
             node[above] {\scriptsize{$b$}}
             (Fbb);

    \path[<->] (Fab) edge[blue]
             node[right] {\scriptsize{$a$}}
             (Fb);

    \path[<->] (Fabb) edge[blue]
             node[right] {\scriptsize{$a$}}
             (Fbb);
   
\end{tikzpicture}
\vspace{-5pt}
\caption{Schreier graph $\Schreier{F,\operatorname{PSL}_2(\Z), \{a,b\}}$} \label{fig: sch F}
\end{figure}

Now, we shall see that $F$ is a free group of rank two, freely generated by the elements 
 $$
p=[a,b]=ab^2ab=\left( \begin{smallmatrix} 2 & -1 \\ -1 & 1 \end{smallmatrix} \right) \qquad \text{and} \qquad q=[b^2,a]=bab^2a=\left( \begin{smallmatrix} 2 & 1 \\ 1 & 1 \end{smallmatrix} \right).
 $$
Clearly, $M=\gen{p,q}\leqslant F\leqslant_6 \operatorname{PSL}_2(\Z)$. So, the cosets of the six representatives considered above but taken modulo $M$, namely $M$, $Mb$, $Mb^2$, $Ma$, $Mab$, and $Mab^2$, are also different from each other. Note that this list of cosets modulo $M$ is possibly not complete (correponding to the fact that $M$ is a subgroup of $F$, possibly proper); in other words, $\Schreier{M,\operatorname{PSL}_2(\Z), \{a,b\}}$ contains these six vertices, but possibly more. However, if we look carefully at each edge of $\Schreier{F,\operatorname{PSL}_2(\Z), \{a,b\}}$, we see that all of them are also correct edges in the graph $\Schreier{M,\operatorname{PSL}_2(\Z), \{a,b\}}$: the seven coset equalities $M\cdot a=Ma$, $M\cdot b=Mb$, $Mb\cdot b=Mb^2$, $Mb^2\cdot b=M$, $Ma\cdot b=Mab$, $Mab\cdot b=Mab^2$, and $Mab^2\cdot b=Ma$ are clear, while the other two, $Mb\cdot a=Mab$ and $Mb^2\cdot a=Mab^2$ are also correct because $bab^{-1}a^{-1}=q\in M$ and $b^2ab^{-2}a^{-1}=p^{-1}\in M$, respectively. This means that the graph depicted in Figure~\ref{fig: sch F} coincides also with $\Schreier{M,\operatorname{PSL}_2(\Z), \{a,b\}}$; therefore, $\gen{p,q}=M=F$. 

It remains to see that $\{p, q\}$ freely generate $F$. This follows, for example, from an application of the Kurosh Subgroup Theorem (see~\cite[Ch.~1, Thm.~7.8]{DD89}): as all subgroups of free products, $F$ must be a free product of a free group and certain conjugates of subgroups of $C_2=\gen{a}$ and $C_3=\gen{b}$, i.e., a free product of a free group and certain conjugates of $C_2=\gen{a}$ and $C_3=\gen{b}$. But $F$ is a normal subgroup of $\operatorname{PSL}_2(\Z)$ which contains neither $a$ nor $b$; hence, it cannot contain any conjugate of them, either. We deduce that $F$ must be a free group. Finally, since $\{p,q\}$ generate $F$, and $F$ is not cyclic ($pq=\left( \begin{smallmatrix} 3 & 1 \\ -1 & 0 \end{smallmatrix}\right) \neq \left( \begin{smallmatrix} 3 & -1 \\ 1 & 0 \end{smallmatrix}\right) =qp$), we deduce that $\{p,q\}$ is, in fact, a free basis for $F\lhd_6 \operatorname{PSL}_2(\Z)$.

Finally, let us exemplify Corollary~\ref{cor: main matrius} by using the short exact sequence~\eqref{eq: ses} with finite quotient, and free kernel freely generated by $\{p, q\}$ (this is a particular materialization of the virtual freeness of $\operatorname{PSL}_2(\Z)$). We develop the following two particular examples.

\begin{ex}
Suppose we are given the matrices $h_1=\left( \begin{smallmatrix} 2 & -1 \\ -1 & 1 \end{smallmatrix}\right)$, $h_2=\left( \begin{smallmatrix} 2 & -5 \\ 1 & -2 \end{smallmatrix}\right)$, and $g=\left( \begin{smallmatrix} 5 & 3 \\ 3 & 2 \end{smallmatrix}\right)$ from $\operatorname{PSL}_2(\Z)$. Let us decide whether $g$ is algebraic or transcendental over $H=\gen{h_1, h_2}\leqslant \operatorname{PSL}_2(\Z)$. 

Writing them in terms of $a,b$ we get $g=bab^2abab^2a$, $h_1=ab^2ab$, and $h_2=babab^2ab^2$ (these expressions can always be found by brute force, knowing in advance that $a,b$ generate $\operatorname{PSL}_2(\Z)$). And, mapping down to $C_2\times C_3$, we get $\overline{h_1}=(0,0)$ and $\overline{h_2}=(1,0)$; so, $\overline{H}=\gen{(1,0)}$. Since $\overline{g}=(0,0)$, 
 $$
I_{\overline{H}}(\overline{g})=\ker \Big( \overline{H}*\gen{x} \twoheadrightarrow \gen{\overline{H}, \overline{g}}=\overline{H}\Big) \leqslant_2 \overline{H}*\gen{x}
 $$
and so, $I_{H}(g;F)=\big(I_{\overline{H}}(\overline{g})\big)\varphi^{-1}\leqslant_2 H*\gen{x}$, where 
 \begin{equation}\label{eeq:phi}
\begin{array}{rcl} \varphi=\pi|_H*id \colon H*\gen{x} & \twoheadrightarrow & \overline{H}*\gen{x}. \\ w(x) & \mapsto & \overline{w}(x) \end{array}
 \end{equation}
Following the algorithm from the proof of Proposition~\ref{prop: contr Sch}, let us construct the Schreier graph $\Schreier{I_H(g;F),\, H*\gen{x},\, \{h_1, h_2, x\}}$ (note that membership in $I_H(g;F)\leqslant_2 H*\gen{x}$ is easily solvable by plugging $g$ into $x$, mapping down to $C_2\times C_3$ and checking whether the result equals $(0,0)$). This graph has just two vertices, which are $I_H(g;F)$ and $I_H(g;F)h_2$, since $\overline{h_2}=(1,0)\neq (0,0)$. After the following easy calculations, 
 $$
\begin{array}{lcl}
I_H(g;F)\cdot h_1 =I_H(g;F) & \text{since} & \overline{h_1}=(0,0), \\
I_H(g;F)\cdot x =I_H(g;F) & \text{since} & \overline{g}=(0,0), \\
I_H(g;F)h_2\cdot h_1 =I_H(g;F)h_2 & \text{since} & \overline{h_2h_1h_2^{-1}}=(0,0), \\
I_H(g;F)h_2\cdot x =I_H(g;F)h_2 & \text{since} & \overline{h_2gh_2^{-1}}=(0,0), \\
I_H(g;F)h_2\cdot h_2 =I_H(g;F) & \text{since} & \overline{h_2}^{\, 2}=(0,0), 
\end{array}
 $$
we obtain the graph $\Schreier{I_H(g;F),\, H*\gen{x},\, \{h_1, h_2, x\}}$, depicted in Figure~\ref{fig: sch eqs}.


\begin{figure}[h]
\centering
  \begin{tikzpicture}[>=stealth, auto, node distance=1cm]
  
   \node[] (A) {$_{I_H(g;F)}$};
   \node[] (B) [right = of A] {$_{I_H(g;F)h_2}$};

   \path[->] (A) edge[line width=1, bend left] node[pos=0.5, above] {$h_2$} (B);

   \path[->] (B) edge[bend left] node[pos=0.5, below] {$h_2$} (A);
   
   \path[->] (A) edge[out=80, in=55, loop,looseness=11] node[pos=0.5, above] {$x$} (A);

   \path[->] (A) edge[out=125, in=100, loop,looseness=11] node[pos=0.5, above] {$h_1$} (A);

   \path[->] (B) edge[out=80, in=55, loop, looseness=11] node[pos=0.5, above] {$x$} (B);

   \path[->] (B) edge[out=125, in=100, loop,looseness=11] node[pos=0.5, above] {$h_1$} (B);

\end{tikzpicture}
\vspace{-5pt}
\caption{Schreier graph $\Schreier{I_H(g;F),\, H*\gen{x},\, \{h_1, h_2, x\}}$} \label{fig: sch eqs}
\end{figure}
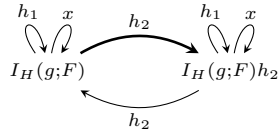

Then, taking as maximal tree the two vertices together with the edge going to the right (boldfaced in the  figure), we get the following set of five $H$-equations generating $I_H(g;F)\leqslant H*\gen{x}$:
 $$
I_H(g;F)=\big\langle w_1(x),\, w_2(x),\, w_3(x),\, w_4(x),\, w_5(x) \big\rangle,
 $$
where $w_1(x)=h_1$, $w_2(x)=x$, $w_3(x)=h_2^2=_H 1$, $w_4(x)=h_2h_1h_2^{-1}$, and $w_5(x)=h_2xh_2^{-1}$. Note that $w_3(x)$ is the trivial equation (so, we can ignore it) and that $w_1(x), w_4(x)\in H$ are constant non-trivial equations. The next step is to evaluate the remaining four equations on $g$, and to rewrite the result in terms of the free generators $\{p,q\}$:
 $$
\begin{array}{l}
v_1=w_1(g)=h_1=p, \\
v_2=w_2(g)=g=q^2, \\
v_4=w_4(g)=h_2h_1h_2^{-1}=qpqp^{-1}q^{-1}p^{-1}q^{-1}, \\
v_5=w_5(g)=h_2gh_2^{-1}=qpq^{-2}p^{-1}q^{-1}.
\end{array}
 $$

Now it is time to apply effective coherence of free groups, and to find a presentation for the subgroup $V=\gen{v_1, v_2, v_4, v_5}=\gen{p, q^2, qpqp^{-1}q^{-1}p^{-1}q^{-1}, qpq^{-2}p^{-1}q^{-1}}\leqslant F=F(\{p,q\})$, with respect to the shown generators. 
To this purpose, we apply the adapted version~\cite[Thm.~5.8]{DV24} of Rosenmann--Ventura theorem~\cite{RV24} mentioned above. The Stallings graph of $V$ (with respect to the ambient free basis $\{p,q\}$) is depicted in Figure~\ref{fig: stallings1}. 


\begin{figure}[h] 
\centering
  \begin{tikzpicture}[>=stealth, auto, node distance=1cm]
  
   \node[state,accepting] (A) {};
   
   \node[state] (B)  [right = of A] {};

   \node[state] (C) [right = of B]{};

   \node[state] (D) [right = of C]{};

   \path[->] (A) edge[out=140, in=100, loop] node[pos=0.5, above] {\scriptsize{$p$}} (A);
   
   \path[->] (A) edge[bend left] node[pos=0.5, above] {\scriptsize{$q$}} (B);

   \path[->] (B) edge[bend left] node[pos=0.5, below] {\scriptsize{$q$}} (A);

   \path[->] (B) edge node[above] {\scriptsize{$p$}} (C);

   \path[->] (C) edge[bend left] node[pos=0.5, above] {\scriptsize{$q$}} (D);

   \path[->] (D) edge[bend left] node[pos=0.5, below] {\scriptsize{$q$}} (C);

   \path[->] (D) edge[out=80, in=40, loop] node[pos=0.5, above] {\scriptsize{$p$}} (D);

  \end{tikzpicture}
  \vspace{-5pt}

  \caption{Stallings graph of $V$} \label{fig: stallings1}
\end{figure}
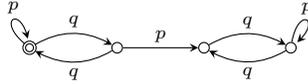

Since its rank equals $1-4+7=4$, we deduce that $\operatorname{rk}(V)=4$ and that $\{v_1, v_2, v_4, v_5\}$ are freely independent. This means that they satisfy no nontrivial relation. We conclude that $I_H(g)=\{1\}$ and the element $g$ is transcendental over $H$.  

We remind the meaning of this conclusion: for $g=\left( \begin{smallmatrix} 5 & 3 \\ 3 & 2 \end{smallmatrix}\right)$, there is no possible election of $n+1\geq 2$ matrices $k_0, k_1,\ldots ,k_n\in H=\langle \left( \begin{smallmatrix} 2 & -1 \\ -1 & 1 \end{smallmatrix}\right),\, \left( \begin{smallmatrix} 2 & -5 \\ 1 & -2 \end{smallmatrix}\right)\rangle\leqslant \operatorname{PSL}_2(\Z)$, $k_1,\ldots ,k_{n-1}\neq I$, and $n$ signums $\epsilon_1,\ldots ,\epsilon_n=\pm 1$, such that $k_0 g^{\epsilon_1} k_1 g^{\epsilon_2} \cdots g^{\epsilon_n} k_n=I$. This is not at all obvious if we use only basic matrix and linear algebra techniques. Observe that this conclusion does not mean that $\{h_1, h_2, g\}$ are freely independent; in fact, they are not since $H$ is not even a free group: $h_2$ is a torsion element, $h_2^2=-I=I$. \qed
\end{ex}

\begin{ex}
Let us develop a similar example with the given matrices $h_1=\left( \begin{smallmatrix} 2 & -1 \\ -1 & 1 \end{smallmatrix}\right)$, $h_2=\left( \begin{smallmatrix} 2 & -5 \\ 1 & -2 \end{smallmatrix}\right)$, and $g=\left( \begin{smallmatrix} 1 & 0 \\ -2 & 1 \end{smallmatrix}\right)$ from $\operatorname{PSL}_2(\Z)$. Again, we have to decide whether $g$ is algebraic or transcendental over $H=\gen{h_1, h_2}$. 

The subgroup is the same as in the previous example, $H=\gen{h_1, h_2}$ with 
$h_1=ab^2ab$, and $h_2=babab^2ab^2$, while now $g=abab$. Mapping down to $C_2\times C_3$, we get $\overline{H}=\gen{(1,0)}$ and $\overline{g}=(0,2)$, so now $\gen{\overline{H},\overline{g}}=C_2\times C_3$. In this situation, 
 $$
I_{\overline{H}}(\overline{g})=\ker \Big( \overline{H}*\gen{x} \twoheadrightarrow \gen{\overline{H}, \overline{g}}=C_2\times C_3\Big) \leqslant_6 \overline{H}*\gen{x}
 $$
and $I_{H}(g;F)=\big(I_{\overline{H}}(\overline{g})\big) \varphi^{-1} \leqslant_6 H*\gen{x}$ (see~\eqref{eeq:phi} for the definition of $\varphi$).

Following the algorithm from the proof of Proposition~\ref{prop: contr Sch}, we get the Schreier graph $\Schreier{I_H(g;F),\, H*\gen{x},\, \{h_1, h_2, x\}}$, depicted in Figure~\ref{fig: sch eqs2} (detailed calculations are left to the reader).


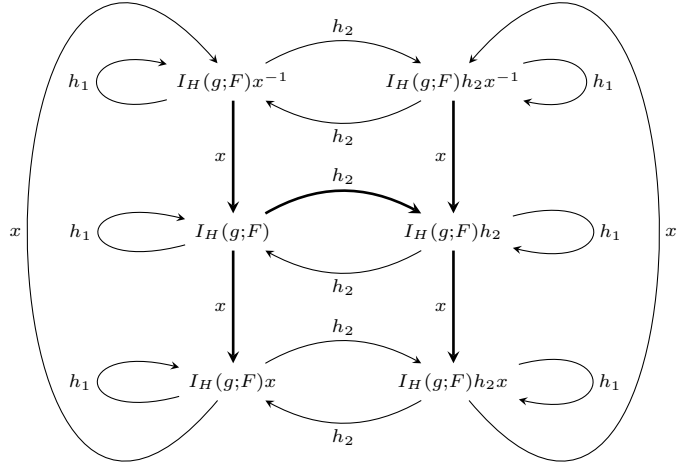
\begin{figure}[H]
\centering
\vspace{-80pt}
  \begin{tikzpicture}[>=stealth, auto, node distance=1.5cm]
   \node[] (gF) {$_{I_H(g;F)}$};
   
   \node[] (gFh2) [right = of gF] {$_{I_H(g;F)h_2}$};

   \node[] (gFx-1) [above = of gF] {$_{I_H(g;F)x^{-1}}$};
   
   \node[] (gFh2x-1) [above = of gFh2] {$_{I_H(g;F)h_2x^{-1}}$};
   
   \node[] (gFx) [below = of gF] {$_{I_H(g;F)x}$};
   
   \node[] (gFh2x) [below = of gFh2] {$_{I_H(g;F)h_2x}$};

   \path[->] (gF) edge[line width=1, bend left] node[pos=0.5, above] {$h_2$} (gFh2);

   \path[->] (gFh2) edge[bend left] node[pos=0.5, below] {$h_2$} (gF);

   \path[->] (gFx-1) edge[bend left] node[pos=0.5, above] {$h_2$} (gFh2x-1);

   \path[->] (gFh2x-1) edge[bend left] node[pos=0.5, below] {$h_2$} (gFx-1);

   \path[->] (gFx) edge[bend left] node[pos=0.5, above] {$h_2$} (gFh2x);

   \path[->] (gFh2x) edge[bend left] node[pos=0.5, below] {$h_2$} (gFx);

   \path[->] (gFx-1) edge[line width=1] node[pos=0.5, left] {$x$} (gF);

   \path[->] (gF) edge[line width=1] node[pos=0.5, left] {$x$} (gFx);

   \path[->] (gFh2x-1) edge[line width=1] node[pos=0.5, left] {$x$} (gFh2);

   \path[->] (gFh2) edge[line width=1] node[pos=0.5, left] {$x$} (gFh2x);

   \path[->] (gFh2x-1) edge[loop right, looseness=6] node[pos=0.5, right] {$h_1$} (gFh2x-1);

   \path[->] (gFh2) edge[loop right, looseness=9] node[pos=0.5, right] {$h_1$} (gFh2);

   \path[->] (gFh2x) edge[loop right, looseness=7.5] node[pos=0.5, right] {$h_1$} (gFh2x);

   \path[->] (gFx-1) edge[loop left, looseness=7] node[pos=0.5, left] {$h_1$} (gFx-1);

   \path[->] (gF) edge[loop left, looseness=12] node[pos=0.5, left] {$h_1$} (gF);

   \path[->] (gFx) edge[loop left, looseness=10] node[pos=0.5, left] {$h_1$} (gFx);

   \path[->] (gFx) edge[bend left, out=140, in=40, looseness=3] node[pos=0.5, left] {$x$} (gFx-1);

   \path[<-] (gFh2x-1) edge[bend right, out=140, in=40, looseness=3] node[pos=0.5, right] {$x$} (gFh2x);
   \end{tikzpicture}
   \vspace{-100pt}
   \caption{Schreier graph $\Schreier{I_H(g;F),\, H*\gen{x},\, \{h_1, h_2, x\}}$} \label{fig: sch eqs2}
\end{figure}

In this case, we have $[H*\gen{x} : I_H(g;F)]=6$ vertices and $6\cdot |\{h_1, h_2, x\}|=18$ edges and so, rank $1-6+18=13$. Choosing the maximal tree indicated by the boldfaced edges, we get the following 13 generators for $I_H(g;F)\leqslant H*\gen{x}$:
 $$
\begin{array}{lll}
w_1(x)=h_1, \qquad & w_2(x)=xh_1 x^{-1}, \qquad & w_3(x)=x^{-1}h_1 x, \\ w_4(x)=h_2x^{-1}h_1xh_2, \qquad & w_5(x)=h_2 h_1 h_2^{-1}, \qquad & w_6(x)=h_2 xh_1 x^{-1} h_2^{-1}, \\ w_7(x)=h_2^2 =1, \qquad & w_8(x)=x^{-1} h_2 x h_2^{-1}, \qquad &  w_9(x)=h_2 x^{-1} h_2 x, \\ w_{10}(x)=xh_2 x^{-1} h_2^{-1}, \qquad & w_{11}(x)=h_2 xh_2 x^{-1}, \qquad & w_{12}(x)=x^3, \\ & w_{13}(x)=h_2 x^3 h_2^{-1}. & 
\end{array}
 $$

Note that $w_7(x)$ is the trivial equation and so we can ignore it, and that $w_1(x), w_5(x)$ are constant non-trivial equations. The next step is to evaluate these twelve equations on $g=abab$, and to rewrite the result in terms of the free generators $\{p,q\}$:
 $$
\begin{array}{rclclcl}
v_1 & = & w_1(g) & = & h_1 & = & p, \\
v_2 & = & w_2(g) & = & gh_1g^{-1} & = & q^{-1}p^{-1}, \\
v_3 & = & w_3(g) & = & g^{-1}h_1 g & = & p^{-1}q\,p, \\
v_4 & = & w_4(g) & = & h_2 g^{-1} h_1 gh_2 & = & q\,p\,q\,p\,q^{-1}p^{-1}q^{-1}p^{-1}q^{-1}, \\
v_5 & = & w_5(g) & = & h_2 h_1 h_2^{-1} & = & q\,p\,q\,p^{-1}q^{-1}p^{-1}q^{-1}, \\
v_6 & = & w_6(g) & = & h_2 g h_1 g^{-1} h_2^{-1} & = & q\,p\,q^2\,p\,q^{-1}p^{-1}q^{-1}, \\
v_8 & = & w_8(g) & = & g^{-1} h_2 gh_2^{-1} & = & p^{-1}q^{-1}p^{-2}q^{-1}p^{-1}q^{-1}p^{-1}q^{-1}, \\
v_9 & = & w_9(g) & = & h_2 g^{-1} h_2 g & = & q\,p\,q\,p\,q\,p^2\,q\,p, \\
v_{10} & = & w_{10}(g) & = & gh_2 g^{-1} h_2^{-1} & = & q^{-4}p^{-1}q^{-1}, \\
v_{11} & = & w_{11}(g) & = & h_2 g h_2 g^{-1} & = & q\,p\,q^4, \\
v_{12} & = & w_{12}(g) & = & g^3 & = & q^{-1}p^{-1}q\,p, \\
v_{13} & = & w_{13}(g) & = & h_2 g^3 h_2^{-1} & = & q\,p\,q^2\,p\,q^{-1}p^{-1}q^{-1}p^{-1}q^{-1}. 
\end{array}
 $$ 
Now we look at the subgroup $V=\gen{v_1, v_2, v_3, v_4, v_5, v_6, v_8, v_9, v_{10}, v_{11}, v_{12}, v_{13}}\leqslant F=F(\{p,q\})$ and need to find a presentation for it \emph{in the shown generators}. Note that we should not eliminate redundant generators, even if obvious like $v_3=v_1^{-2}v_2^{-1}v_1$, because we may possibly lose non-trivial equations produced by them in the next step. According to the algorithm from the proof of Proposition~\ref{prop: contr Sch}, this requires to draw a flower automaton with twelve petals corresponding to these twelve generators, fold it until getting the corresponding $\{p,q\}$-Stallings automaton, and then analyze each one of the possible closed folding occurring along the tower ($g$ will be $H$-transcendental if and only if no such closed folding occurs). This is a long procedure which we can greatly abbreviate after realizing that the first two generators, $v_1=p$ and $v_2=q^{-1}p^{-1}$, already generate the full group $F$; in other words, the first two petals already fold down to the $\{p,q\}$-bouquet (with no closed folding). Therefore, each one of the remaining 10 petals will provoke a unique closed folding in the tower, corresponding to the relation expressing that generator in terms of $p=v_1$ and $q=v_1^{-1}v_2^{-1}$. Proceeding in this way, we get a full set of relations satisfied by $\{v_1, v_2, v_3, v_4, v_5, v_6, v_8, v_9, v_{10}, v_{11}, v_{12}, v_{13}\}$:
 $$
\begin{array}{rcl}
v_1^{-2} v_2^{-1} v_1 v_3^{-1} & = & 1, \vspace{.2cm} \\
v_1^{-1} v_2^{-2} v_1 v_2^3 v_1 v_4^{-1} & = & 1, \vspace{.2cm} \\
v_1^{-1} v_2^{-2} v_1^{-1} v_2^2 v_1 v_5^{-1} & = & 1, \vspace{.2cm} \\
v_1^{-1} v_2^{-2} v_1^{-1} v_2^{-1} v_1 v_2^2 v_1 v_6^{-1} & = & 1, \vspace{.2cm} \\
v_1^{-1} v_2 v_1^{-1} v_2^3 v_1 v_8^{-1} & = & 1, \vspace{.2cm} \\ 
v_1^{-1} v_2^{-3} v_1 v_2^{-1} v_1 v_9^{-1} & = & 1, \vspace{.2cm} \\
v_2 v_1 v_2 v_1 v_2 v_1 v_2^2 v_1 v_{10}^{-1} & = & 1, \vspace{.2cm} \\
v_1^{-1} v_2^{-2} v_1^{-1} v_2^{-1} v_1^{-1} v_2^{-1} v_1^{-1} v_2^{-1} v_{11}^{-1} & = & 1, \vspace{.2cm} \\
v_2 v_1^{-1} v_2^{-1} v_1 v_{12}^{-1} & = & 1, \vspace{.2cm} \\
v_1^{-1} v_2^{-2} v_1^{-1} v_2^{-1} v_1 v_2^3 v_1 v_{13}^{-1} & = & 1.
\end{array}
 $$
Finally, replacing $v_i$ by $w_i(x)$ in each one of these expressions, we get a collection of $H$-equations generating the ideal $I_H(g)$ (as normal subgroup of $H*\gen{x}$):
 $$
\begin{array}{c}
h_1^{-2} x h_1^{-1} x^{-1} h_1 x^{-1} h_1^{-1} x, \vspace{.2cm} \\
h_1^{-1} x h_1^{-2} x^{-1} h_1 x h_1^3 x^{-1} h_1 h_2^{-1} x^{-1} h_1^{-1} x h_2^{-1}, \vspace{.2cm} \\
h_1^{-1} x h_1^{-2} x^{-1} h_1^{-1} x h_1^2 x^{-1} h_1 h_2 h_1^{-1} h_2^{-1}, \vspace{.2cm} \\
h_1^{-1} x h_1^{-2} x^{-1} h_1^{-1} x h_1^{-1} x^{-1} h_1 x h_1^2 x^{-1} h_1 h_2 x h_1^{-1} x^{-1} h_2^{-1}, \vspace{.2cm} \\
h_1^{-1} x h_1 x^{-1} h_1^{-1} x h_1^3 x^{-1} h_1 h_2 x^{-1} h_2^{-1} x, \vspace{.2cm} \\
h_1^{-1} x h_1^{-3} x^{-1} h_1 x h_1^{-1} x^{-1} h_1 x^{-1} h_2^{-1} x h_2^{-1}, \vspace{.2cm} \\
h_1 x^{-1} h_1 x h_1 x^{-1} h_1 x h_1 x^{-1} h_1 x h_1^2 x^{-1} h_1 h_2 x h_2^{-1}, \vspace{.2cm} \\
h_1^{-1} x h_1^{-2} x^{-1} h_1^{-1} x h_1^{-1} x^{-1} h_1^{-1} x h_1^{-1} x^{-1} h_1^{-1} x h_1^{-1} h_2^{-1} x^{-1} h_2^{-1}, \vspace{.2cm} \\
h_1 x^{-1} h_1^{-1} x h_1^{-1} x^{-1} h_1 x^{-2}, \vspace{.2cm} \\
h_1^{-1} x h_1^{-2} x^{-1} h_1^{-1} x h_1^{-1} x^{-1} h_1 x h_1^3 x^{-1} h_1 h_2 x^{-3} h_2^{-1}. 
\end{array}
 $$
These are ten $H$-equations satisfied by the matrix $g=\left( \begin{smallmatrix} 1 & 0 \\ -2 & 1 \end{smallmatrix}\right)$, and in such a way that any $w(x)\in H*\gen{x}$ with $w(g)=I$ is a product of conjugates of them. In matrix form (for aesthetic reasons, we write $X$ instead of $x$) they read as follows: 
$$
\begin{array}{c}
\left( \begin{smallmatrix} 2 & 3 \\ 3 & 5 \end{smallmatrix}\right) X \left( \begin{smallmatrix} 1 & 1 \\ 1 & 2 \end{smallmatrix}\right) X^{-1}\left( \begin{smallmatrix} 2 & -1 \\ -1 & 1 \end{smallmatrix}\right) X^{-1}\left( \begin{smallmatrix} 1 & 1 \\ 1 & 2 \end{smallmatrix}\right) X=I, 

\\ \\

\left( \begin{smallmatrix} 1 & 1 \\ 1 & 2 \end{smallmatrix}\right) X \left( \begin{smallmatrix} 2 & 3 \\ 3 & 5 \end{smallmatrix}\right) X^{-1}\left( \begin{smallmatrix} 2 & -1 \\ -1 & 1 \end{smallmatrix}\right) X\left( \begin{smallmatrix} 13 & -8 \\ -8 & 5 \end{smallmatrix}\right) X^{-1}\left( \begin{smallmatrix} -3 & 8 \\ 1 & -3 \end{smallmatrix}\right) X^{-1}\left( \begin{smallmatrix} 1 & 1 \\ 1 & 2 \end{smallmatrix}\right) X\left( \begin{smallmatrix} -2 & 5 \\ -1 & 2 \end{smallmatrix}\right)=I,

\\ \\

\left( \begin{smallmatrix} 1 & 1 \\ 1 & 2 \end{smallmatrix}\right) X\left( \begin{smallmatrix} 2 & 3 \\ 3 & 5 \end{smallmatrix}\right) X^{-1}\left( \begin{smallmatrix} 1 & 1 \\ 1 & 2 \end{smallmatrix}\right) X\left( \begin{smallmatrix} 5 & -3 \\ -3 & 2 \end{smallmatrix}\right) X^{-1}\left( \begin{smallmatrix} 23 & -51 \\ -9 & 20 \end{smallmatrix}\right)=I,

\\ \\

\left( \begin{smallmatrix} 1 & 1 \\ 1 & 2 \end{smallmatrix}\right) X \left( \begin{smallmatrix} 2 & 3 \\ 3 & 5 \end{smallmatrix}\right) X^{-1}\left( \begin{smallmatrix} 1 & 1 \\ 1 & 2 \end{smallmatrix}\right) X\left( \begin{smallmatrix} 1 & 1 \\ 1 & 2 \end{smallmatrix}\right) X^{-1}\left( \begin{smallmatrix} 2 & -1 \\ -1 & 1 \end{smallmatrix}\right) X\left( \begin{smallmatrix} 5 & -3 \\ -3 & 2 \end{smallmatrix}\right) X^{-1}\left( \begin{smallmatrix} 3 & -8 \\ -1 & 3 \end{smallmatrix}\right)X\left( \begin{smallmatrix} 1 & 1 \\ 1 & 2 \end{smallmatrix}\right) X^{-1}\left( \begin{smallmatrix} -2 & 5 \\ -1 & 2 \end{smallmatrix}\right)=I,

\\ \\

\left( \begin{smallmatrix} 1 & 1 \\ 1 & 2 \end{smallmatrix}\right) X \left( \begin{smallmatrix} 2 & -1 \\ -1 & 1 \end{smallmatrix}\right) X^{-1}\left( \begin{smallmatrix} 1 & 1 \\ 1 & 2 \end{smallmatrix}\right) X\left( \begin{smallmatrix} 13 & -8 \\ -8 & 5 \end{smallmatrix}\right) X^{-1}\left( \begin{smallmatrix} 3 & -8 \\ -1 & 3 \end{smallmatrix}\right) X^{-1}\left( \begin{smallmatrix} -2 & 5 \\ -1 & 2 \end{smallmatrix}\right) X=I,

\\ \\

\left( \begin{smallmatrix} 1 & 1 \\ 1 & 2 \end{smallmatrix}\right) X \left( \begin{smallmatrix} 5 & 8 \\ 8 & 13 \end{smallmatrix}\right) X^{-1}\left( \begin{smallmatrix} 2 & -1 \\ -1 & 1 \end{smallmatrix}\right) X\left( \begin{smallmatrix} 1 & 1 \\ 1 & 2 \end{smallmatrix}\right) X^{-1}\left( \begin{smallmatrix} 2 & -1 \\ -1 & 1 \end{smallmatrix}\right) X^{-1}\left( \begin{smallmatrix} -2 & 5 \\ -1 & 2 \end{smallmatrix}\right) X \left( \begin{smallmatrix} -2 & 5 \\ -1 & 2 \end{smallmatrix}\right)=I,

\\ \\

\left[\left( \begin{smallmatrix} 2 & -1 \\ -1 & 1 \end{smallmatrix}\right) X^{-1} \left( \begin{smallmatrix} 2 & -1 \\ -1 & 1 \end{smallmatrix}\right) X\right]^3\left( \begin{smallmatrix} 5 & -3 \\ -3 & 2 \end{smallmatrix}\right) X^{-1}\left( \begin{smallmatrix} 3 & -8 \\ -1 & 3 \end{smallmatrix}\right) X \left( \begin{smallmatrix} -2 & 5 \\ -1 & 2 \end{smallmatrix}\right)=I,

\\ \\

\left( \begin{smallmatrix} 1 & 1 \\ 1 & 2 \end{smallmatrix}\right) X \left( \begin{smallmatrix} 2 & 3 \\ 3 & 5 \end{smallmatrix}\right) X^{-1}\left( \begin{smallmatrix} 1 & 1 \\ 1 & 2 \end{smallmatrix}\right) X\left( \begin{smallmatrix} 1 & 1 \\ 1 & 2 \end{smallmatrix}\right) X^{-1}\left( \begin{smallmatrix} 1 & 1 \\ 1 & 2 \end{smallmatrix}\right) X\left( \begin{smallmatrix} 1 & 1 \\ 1 & 2 \end{smallmatrix}\right) X^{-1}\left( \begin{smallmatrix} 1 & 1 \\ 1 & 2 \end{smallmatrix}\right) X\left( \begin{smallmatrix} -3 & 7 \\ -4 & 9 \end{smallmatrix}\right)X^{-1}\left( \begin{smallmatrix} -2 & 5 \\ -1 & 2 \end{smallmatrix}\right)=I,
 
\\ \\

\left( \begin{smallmatrix} 2 & -1 \\ -1 & 1 \end{smallmatrix}\right) X^{-1} \left( \begin{smallmatrix} 1 & 1 \\ 1 & 2 \end{smallmatrix}\right) X\left( \begin{smallmatrix} 1 & 1 \\ 1 & 2 \end{smallmatrix}\right) X^{-1}\left( \begin{smallmatrix} 2 & -1 \\ -1 & 1 \end{smallmatrix}\right) X^{-2}=I,

\\ \\

\left( \begin{smallmatrix} 1 & 1 \\ 1 & 2 \end{smallmatrix}\right) X \left( \begin{smallmatrix} 2 & 3 \\ 3 & 5 \end{smallmatrix}\right) X^{-1}\left( \begin{smallmatrix} 1 & 1 \\ 1 & 2 \end{smallmatrix}\right) X\left( \begin{smallmatrix} 1 & 1 \\ 1 & 2 \end{smallmatrix}\right) X^{-1}\left( \begin{smallmatrix} 2 & -1 \\ -1 & 1 \end{smallmatrix}\right) X\left( \begin{smallmatrix} 13 & -8 \\ -8 & 5 \end{smallmatrix}\right) X^{-1}\left( \begin{smallmatrix} 3 & -8 \\ -1 & 3 \end{smallmatrix}\right) X^{-3}\left( \begin{smallmatrix} -2 & 5 \\ -1 & 2 \end{smallmatrix}\right)=I.
\end{array}
 $$

\noindent Again, it is not at all obvious how to arrive at this conclusion using only basic matrix and linear algebra techniques. \qed
\end{ex}

We finish this note by highlighting that the groups of integral matrices of size 4 and bigger behave very differently with respect to algorithmic properties. It is well known that, for $n\geq 4$, and $G$ being each of $\operatorname{GL}_n(\Z)$, $\operatorname{PGL}_n(\Z)$, $\operatorname{SL}_n(\Z)$, or $\operatorname{PSL}_n(\Z)$, the group $\F_2\times \F_2$ embeds in $G$. As a consequence, non of these groups is coherent. And, by Proposition~\ref{prop:equivalence}, they are not eq-coherent either. This means the following:

\begin{cor}
Let $G$ be one of $\operatorname{GL}_n(\Z)$, $\operatorname{PGL}_n(\Z)$, $\operatorname{SL}_n(\Z)$, or $\operatorname{PSL}_n(\Z)$, with $n\geq 4$. There exist matrices $h_1,\ldots ,h_s;g\in G$ such that $g$ is algebraic over $H=\gen{h_1,\ldots ,h_s}\leqslant G$ but the ideal $I_H(g)$ \emph{is not finitely generated}, even as normal subgroup of $H*\gen{x}$. \qed
\end{cor}

Similar statements are also true for size two matrices over more complicated rings. For example, according to~\cite[Page 734]{S74}, G. Baumslag observed that $\operatorname{SL}_2(\Z [1/n])$ is not coherent as soon as $n$ is divisible by two different primes $p,q$, with the subgroup $\gen{\left( \begin{smallmatrix} 1 & 1 \\ 0 & 1 \end{smallmatrix}\right),\, \left( \begin{smallmatrix} p/q & 0 \\ 0 & q/p \end{smallmatrix}\right)}\leqslant \operatorname{SL}_2(\Z [1/n])$ being finitely generated but not finitely presented. 

As usual (see~\cite{BMV10} for a different situation where the same phenomenon happens), the case of size 3 integral matrices is not known: these groups are known not to be virtually free (so, the arguments in the present note do not apply to them), but they are also known not to contain $\F_2\times \F_2$ either; see~\cite[Prop.~2.3]{RiZa}. As far as we are aware of, the following question asked by J.P.~Serre in 1979 under the coherence form (see~\cite[Page 734]{S74}) remains open almost fifty years later:

\begin{que}
Let $G$ be $\operatorname{GL}_3(\Z)$, $\operatorname{PGL}_3(\Z)$, $\operatorname{SL}_3(\Z)$, or $\operatorname{PSL}_3(\Z)$. Is $G$ effectively eq-coherent? Do there exist matrices $h_1,\ldots ,h_s;g\in G$ such that $g$ is algebraic over $H=\gen{h_1,\ldots ,h_s}\leqslant G$ but the ideal $I_H(g)$ \emph{is not} finitely generated, even as normal subgroup of $H*\gen{x}$? If not, is there an algorithm to compute generators for $I_H(g)$ as normal subgroup of $H*\gen{x}$? 
\end{que}

\section*{Acknowledgements}

\noindent The authors acknowledge support from the Spanish Agencia Estatal de Investigaci\'on through grant PID2021-126851NB-100 (AEI/ FEDER, UE).

\end{document}